\documentclass[12pt,british,a4paper]{amsart}

\usepackage{lmodern}
\usepackage[T1]{fontenc}
\usepackage{microtype}

\usepackage{babel,amssymb,amscd,url,stmaryrd,mathtools,fixltx2e}
\usepackage{paralist,csquotes,stackrel}
\usepackage[mathscr]{eucal}

\numberwithin{equation}{section}
\mathtoolsset{mathic=true}
\setdefaultenum{\upshape (i)}{}{}{}

\theoremstyle{plain}

\newtheorem{lemma}[equation]{Lemma}

\theoremstyle{definition}
\newtheorem{definition}[equation]{Definition}

\theoremstyle{remark}
\newtheorem{Example}[equation]{Example}
\newtheorem{remark}[equation]{Remark}

\newenvironment{example}{\begin{Example}\pushQED{\qee}}{\popQED\end{Example}}
\DeclareRobustCommand{\qee}{%
  \ifmmode \mathqee
  \else
    \leavevmode\unskip\penalty9999 \hbox{}\nobreak\hfill
    \quad\hbox{\qeesymbol}%
  \fi
}
\newcommand{\mathqee}{\quad\hbox{\qeesymbol}}
\newcommand{\qeesymbol}{\ensuremath\diamondsuit}

\newcommand{\Lie}[1]{\operatorname{\textsl{#1}}}
\newcommand{\lie}[1]{\operatorname{\mathfrak{#1}}}

\newcommand{\bmf}{\lie b}

\newcommand{\kf}{\lie k}
\newcommand{\n}{\lie n}

\newcommand{\SL}{\Lie{SL}}
\newcommand{\SU}{\Lie{SU}}
\newcommand{\Un}{\Lie{U}}

\newcommand{\C}{{\mathbb C}}
\newcommand{\HH}{{\mathbb H}}

\newcommand{\cO}{\mathcal O}
\newcommand{\tH}{\tilde H}

\newcommand{\symp}{{\sslash}} 
\newcommand{\hkq}{{\sslash\mkern-6mu/}}

\DeclareMathOperator{\Hom}{Hom}

\DeclareMathOperator{\im}{im}

\begin{document}

\title{A multiplicative analogue of complex symplectic implosion}

\author[A.~Dancer]{Andrew Dancer}
\address[Dancer]{Jesus College\\
Oxford\\
OX1 3DW\\
United Kingdom} \email{dancer@maths.ox.ac.uk}

\author[F.~Kirwan]{Frances Kirwan}
\address[Kirwan]{Balliol College\\
Oxford\\
OX1 3BJ\\
United Kingdom} \email{kirwan@maths.ox.ac.uk}


\begin{abstract}
We introduce a multiplicative version of complex-symplectic
implosion in the case of $SL(n, \C)$.
 The universal multiplicative implosion for $SL(n, \C)$ 
is an affine variety and 
can be viewed as a nonreductive
geometric invariant theory quotient. It carries a torus action.
and reductions by this action give the Steinberg fibres of
$SL(n, \C)$. We also explain how the real symplectic group-valued universal
implosion introduced by Hurtubise, Jeffrey and Sjamaar may be
identified inside this space.
\end{abstract}

\maketitle

\section{Introduction}

In a series of papers \cite{DKS,DKS-Seshadri,DKS-twistor,
DKS-Arb,DKR} we
investigated the notion of a universal hyperk\"ahler implosion for a 
compact group $K$,
by analogy with the universal symplectic implosion of \cite{GJS:implosion}.

We recall that the universal symplectic implosion of $K$ is a space 
$(T^*K)_{\rm impl}$ with a Hamiltonian $K \times T$ action (where $T$ is
a maximal torus of $K$), such that
the reductions by $T$ at points in the closed positive Weyl
chamber give the coadjoint orbits of $K$. These orbits are the reductions
of $T^*K$ by the right $K$ factor in the $K \times K$ action on $T^*K$.
The implosion of a general symplectic manifold $X$ with Hamiltonian $K$-action
is obtained by reducing $X \times (T^*K)_{\rm impl}$ by the diagonal $K$ action,
producing a space $X_{\rm impl}$ with $T$ action. The reduction of $X$ by $K$, at any element $\xi$ of a chosen positive Weyl chamber in the dual $\kf^*$ of the Lie algebra of $K$, coincides with the reduction of $X_{\rm impl}$ by $T$ at $\xi$.
In this sense the implosion abelianises the $K$ action on $X$.

In \cite{DKS} we constructed an analogue of implosion for hyperk\"ahler geometry when $K=\SU(n)$.
As a stratified complex-symplectic space
 the universal hyperk\"ahler implosion is
the geometric invariant theory (GIT) quotient
\( (K_\C \times \n^0) \symp N \) where \( N \) is a 
maximal unipotent subgroup of the complexified group \( K_\C \), and
\( \n^0 \) is the annihilator in \( \kf_{\C}^* \) of the Lie algebra
\( \n \) of $N$. The implosion is thus the complex-symplectic quotient, 
in the GIT sense, of
\( T^* K_\C \) by \( N \), just as the symplectic implosion is (as
explained in \cite{GJS:implosion})
the GIT quotient of \( K_\C \) by \( N \). Note that $N$ is nonreductive, so
some work is needed to show that  the quotient exists as an affine variety.
This was shown in the case $K=SU(n)$ in \cite{DKS} and in general
follows from results of Ginzburg-Riche \cite{GinzburgRiche}
(see the discussion in \cite{DKS-Arb}). 

The universal hyperk\"ahler implosion 
carries an action of \( K_\C \times T_\C \) where  \( T \)
is the standard maximal torus of $K$. The presence of
non-semisimple elements in $\kf_\C$, and the fact that non-semisimple
orbits are not closed, means that the abelianisation picture
becomes more complicated than in the real symplectic case.
The complex-symplectic quotients
by the torus action are now
the Kostant varieties; that is, the varieties in \( \kf_{\C}^* \)
obtained by fixing the values of the invariant polynomials for this
Lie algebra \cite{Chriss-G:representation, Kostant:polynomial}.  
The Kostant varieties are unions of complex coadjoint orbits.
The smooth locus of a Kostant variety may be identified with the 
corresponding regular orbit,
 which is open and dense in the Kostant variety with complement
of codimension at least 2.

In \cite{DKS} we considered the case when \( K=\SU(n) \).
In this situation the universal hyperk\"ahler implosion can be identified with 
a hyperk\"ahler quotient using quiver diagrams, and thus
can be seen to be genuinely a stratified hyperk\"ahler space rather than just a 
complex-symplectic one. One may, by analogy with the symplectic case, then
implode a general space with hyperk\"ahler $SU(n)$ action by taking its
product with the universal implosion and performing the hyperk\"ahler
reduction by the diagonal $K$ action.
For general compact groups a direct construction
of a hyperk\"ahler metric on the nonreductive quotient 
\( (K_\C \times \n^0) \symp N \) is not yet available, although in \cite{DKR} we
gave an alternative approach to hyperk\"ahler implosion
via moduli spaces of solutions to Nahm's equations.

Several authors have explored multiplicative quiver diagrams, for
example \cite{Boalch2,CBS,Y}. In this paper we
consider multiplicative analogues of the quiver spaces considered in
\cite{DKS} and obtain an analogue of the universal
hyperk\"ahler implosion for $K=\SU(n)$ in the
quasi-Hamiltonian setting. We obtain a moduli space of solutions to
multiplicative quiver equations that may be identified with the
nonreductive quotient $(SL(n,\C) \times B) \symp N$ where $B$ is the
standard Borel subgroup of $SL(n, \C)$.  This quotient space admits
actions of $K_\C$ and of the torus $T_\C$. The reductions by the torus
action give the Steinberg fibres that are the multiplicative version
of the Kostant varieties. We also show how the real symplectic universal
group-valued implosion of \cite{HJS} may be identified with a
stratified set sitting inside our complex space.

The geometric quotient
 $G \times_U P$, where $G$ is a complex reductive group and $P$
is a parabolic subgroup with unipotent radical $U$,
 also arises in the work of Boalch \cite{Boalch1} as a quasi-Hamiltonian space
for the action of $G \times L$ where $L$ is the Levi factor of $P$
(see also \cite{EL} for some related ideas).
In particular, if $P$ is a Borel subgroup, Boalch obtains the geometric
quotient $G \times_N B$ as a quasi-Hamiltonian $G \times T_\C$ space.
If $G= SL(n, \C)$, we obtain this space as a quasi-affine variety
inside the implosion $(SL(n, \C) \times B) \symp N$ (see the discussion at the 
end of \S 3 and before Definition 4.1).

Pavel Safronov has recently informed us of his work \cite{Saf} which constructs versions of
implosion, as stacks, for general complex semisimple groups in both the Hamiltonian and quasi-Hamiltonian
settings.
 
\subsubsection*{Acknowledgements.} 
We thank Philip Boalch, Kevin McGerty 
and Pavel Safronov for valuable conversations.

\section{Hyperk\"ahler quiver diagrams}
\label{sec:hyperk-quiv-diagr}

Let us recall the 
finite-dimensional approach via quiver diagrams used to
construct the universal hyperk\"ahler implosion for
$K= SU(n)$ in \cite{DKS}.
We started with the
flat hyperk\"ahler space
\begin{equation}
  \label{eq:Mn}
  M = M(\mathbf n) = \bigoplus_{i=1}^{r-1} \HH^{n_i n_{i+1}} = 
  \bigoplus_{i=1}^{r-1} \Hom(\C^{n_i},\C^{n_{i+1}}) \oplus
  \Hom(\C^{n_{i+1}},\C^{n_i})
\end{equation}
with the hyperk\"ahler action of \( \Un(n_1) \times \dots \times
\Un(n_r) \)
\begin{equation*}
  \alpha_i \mapsto g_{i+1} \alpha_i g_i^{-1},\quad
  \beta_i \mapsto g_i \beta_i g_{i+1}^{-1} \qquad (i=1,\dots r-1),
\end{equation*}
with \( g_i \in \Un(n_i) \) for \( i=1, \dots, r \).
Here ${\bf n}$ is the dimension vector $(n_1, \ldots, n_r=n)$  

We took the hyperk\"ahler quotient of \( M(\mathbf n) \) by
the group \( H = \prod_{i=1}^{r-1}\SU(n_i) \), obtaining a stratified
hyperk\"ahler space \( Q = M \hkq H \), with a residual action of
the torus \( T^{r-1} = \tH / H \) where \( \tH = \prod_{i=1}^{r-1}\Un(n_i) \), 
as well as a commuting action of \( \SU(n_r) = \SU(n) \). 

  The \emph{universal hyperk\"ahler implosion for \( \SU(n) \)} was defined
to  be the hyperk\"ahler quotient \( Q = M \hkq H \), where \( M \), \(
  H \) are as above with $r=n$ and \( n_j = j \), for \( j=1, \dots, n
  \), (i.e.\ the case of a {\em full flag quiver}).

As a complex-symplectic space, \( Q \) (for general dimension vector)
is  the GIT quotient, by the 
complexification
\begin{equation*}
  H_\C  = \prod_{i=1}^{r-1}\SL(n_i,\C)
\end{equation*}
of \( H \), of the zero locus of the complex moment map \( \mu_{\C} \) for the
\( H \) action. 

The complex moment map equation \( \mu_\C =0 \) is equivalent to the
equations
\begin{equation}
  \label{eq:mmcomplex}
   \beta_{i+1} \alpha_{i+1} - \alpha_{i}\beta_{i}  = \lambda^\C_{i+1} I \qquad
  (i=0,\dots,r-2),
\end{equation}
for (free) complex scalars \( \lambda^\C_1,\dots,\lambda^\C_{r-1} \).

These equations are invariant under the
 action of \( H_\C \) given by
\begin{gather*}
  \alpha_i \mapsto g_{i+1} \alpha_i g_i^{-1}, \quad \beta_i \mapsto
  g_i \beta_i g_{i+1}^{-1} \qquad (i=1,\dots r-2),\\
  \alpha_{r-1} \mapsto \alpha_{r-1} g_{r-1}^{-1}, \quad \beta_{r-1}
  \mapsto g_{r-1} \beta_{r-1},
\end{gather*}
where \( g_i \in \SL(n_i,\C) \).  

The action of \( \SL(n,\C) =
\SL(n_r,\C) \) on the quotient \( Q \) is given by
\begin{equation*}
  \alpha_{r-1} \mapsto g_r \alpha_{r-1}, \quad
  \beta_{r-1} \mapsto \beta_{r-1} g_r^{-1}.
\end{equation*}

There is also a residual action of \( \tH_\C/H_\C \) which we can
identify, in the full flag case, with the maximal torus \( T_\C \) of
\( K_\C \). The complex numbers \( \lambda_i \) combine to give the
complex-symplectic moment map for this complex torus action.  

It is often useful to consider the endomorphism
\begin{equation*}
  X = \alpha_{r-1} \beta_{r-1} \in \Hom (\C^n,\C^n),
\end{equation*}
which is invariant under the action of \( \tH_\C \) and transforms by
conjugation under the residual \( \SL(n,\C) \) action.

\section{Multiplicative diagrams}
\label{sec:multiplicative}
Let us now consider the multiplicative version of the quiver diagrams
above. That is, we consider the quasi-Hamiltonian 
moment map equations for
the action of $H_\C  = \prod_{i=1}^{r-1}\SL(n_i,\C)$. We refer to
\cite{AMM} for general background on quasi-Hamiltonian spaces.
A result of Van den Bergh \cite{VDB1}, \cite{VDB2}, shows that for
length one quivers
\[
  V \stackrel[\beta]{\alpha}{\rightleftarrows} W,
\]
with $1 + \alpha \beta, 1 + \beta \alpha$ invertible,
the natural $GL(V) \times GL(W)$ action 
\[
(\alpha, \beta) \mapsto (g_2 \alpha g_1^{-1}, g_1 \beta g_2^{-1})
\]
is quasi-Hamiltonian with 
group-valued moment map
\[
(\alpha, \beta) \mapsto (1 + \beta \alpha, (1 + \alpha \beta)^{-1}).
\]
For general quivers of the kind considered in the preceding section
we have an action of $\tilde{H}_{\C} \times GL(n_r, \C)
= \prod_{i=1}^{r} GL(n_i, \C)$. We let $M_{\rm mult} ({\bf n})$ denote
the space of such quivers, with dimension vector
$\bf n$, such that the endomorphisms $1+ \alpha_i \beta_i$
and $1 + \beta_{i} \alpha_i$ are invertible for each $i$.

We shall consider the reduced space by the action of $H_\C =
\prod_{i=1}^{r-1} SL(n_i, \C)$.
The equations the quiver has to satisfy are now
\begin{equation} \label{eq:mult}
(1 + \beta_{i+1} \alpha_{i+1}) = q_{i+1} (1 + \alpha_i \beta_i)  
\end{equation}
for free complex scalars $q_{i+1} \; (0 \leq i \leq r-2)$ . 
(See, for example, \cite{Y} for the associated equations
with $q_i$ fixed, that arise as moment maps for the
$\tilde{H}_\C$ action). Our invertibility conditions mean that the scalars
$q_i$ are all nonzero.

We remark that if all the $q_i$ equal $1$ then we get 
the same equations $\alpha_i \beta_i = \beta_{i+1} \alpha_{i+1}$
as in the additive case with $\lambda_i=0$.

We now define the multiplicative analogue of the hyperk\"ahler spaces
$Q$ of \S 2.

\medskip
\begin{definition}
Let $Q_{\rm mult}({\bf n})$ denote
the GIT quotient by the reductive group $H_\C$ of the 
space of solutions to (\ref{eq:mult}) in $M_{\rm mult}({\bf n})$.
\end{definition}

\begin{remark}
Notice that there is a residual action of 
$GL(n, \C) \times \tilde{H}_{\C}/H_{\C}$
on $Q_{\rm mult}({\bf n})$. In the
full flag case ${\bf n} = (1,2, \ldots, n)$ we may identify 
the complex torus  $\tilde{H}_{\C}/H_{\C} $ with
the maximal torus $T_\C$ in $SL(n, \C)$.
\end{remark}

\medskip
In the full flag case $Q_{\rm mult}({\rm n})$ will be a first approximation
to the multiplicative implosion. The true implosion will be a slight modification of this space, involving passing to a cover at a suitable 
stage of the construction.

\medskip
We now collect some useful results about the multiplicative quiver equations
(for general dimension vectors $\bf n$ unless otherwise stated).
We first consider the endomorphism $Y= 1 + \alpha_{r-1} \beta_{r-1}$. This
is (up to inversion) 
the value of the moment map for the residual $GL(n, \C)$ action
and is the multiplicative analogue of the
endomorphism $X = \alpha_{r-1} \beta_{r-1}$ mentioned above.
\begin{lemma}
$Y = 1 + \alpha_{r-1} \beta_{r-1}$ satisfies the equation
\[
(Y-1)(Y-q_{r-1}) \ldots (Y - q_{r-1} \ldots q_1) =0
\]
\end{lemma}

\begin{proof}
We let $X_k = \alpha_{r-1} \alpha_{r-2} \ldots \alpha_{r-k}
\beta_{r-k} \ldots \beta_{r-2} \beta_{r-1}$ and $X = X_1 =
\alpha_{r-1} \beta_{r-1}$.

Using the equation repeatedly it is now easy to show
that
\[
X_k X = (q_{r-1} \ldots q_{r-k} -1) X_k + q_{r-1}\ldots q_{r-k} X_{k+1}
\]
for $1 \leq k \leq r-1$ (interpreting $X_r$ as $0$).
We deduce
\[
X(X+1 - q_{r-1})(X+ 1- q_{r-1}q_{r-2}) \ldots (X+1 - q_{r-1} \ldots q_1)
=0
\]
which on setting $Y=1+X$ yields the result. 
\end{proof}

In the case when all $q_i$ are $1$, then $Y$ lies in the unipotent
variety.

\begin{remark}
Let us observe that, using our equations, we have:
\[
\beta_i (1 + \alpha_i \beta_i -\tau) = q_i (1 + \alpha_{i-1} \beta_{i-1}
- \tau q_i^{-1} ) \beta_i 
\]
and
\[
(1+ \alpha_i \beta_i -\tau) \alpha_i = q_i \alpha_i 
( 1 + \alpha_{i-1} \beta_{i-1} -\tau q_i^{-1})
\]
It follows that $\alpha_j, \beta_j$ preserve the decomposition of the
quiver into subquivers given by generalised eigenspaces. Explicitly, we have
\begin{equation} \label{summand}
\ker (1 + \alpha_{i-1} \beta_{i-1} -\tau q_i^{-1})^m \stackrel[\beta_i]{\alpha_i}
{\rightleftarrows} \ker (1 + \alpha_{i} \beta_{i} -\tau)^m.
\end{equation}
Notice that 
\[
\ker (1 + \alpha_{i-1} \beta_{i-1} - \tau q_{i}^{-1})^m =
\ker (1 + \beta_i \alpha_i - \tau)^m
\]
using our equations (\ref{eq:mult}) and the fact that $q_i$ are nonzero.
So the maps $\alpha_i, \beta_i$ in (\ref{summand}) 
are isomorphisms unless $\tau = 1$.
\end{remark}

\begin{remark}
As in the additive case we see that $\tilde{H}_\C = 
\prod_{i=1}^{r-1} GL(n_i, \C)$ acts freely on a quiver if, for each $i$,
either $\alpha_i$ is injective or $\beta_i$ is surjective.
If for each $i$, both conditions hold, then the quiver is stable
for the $\tilde{H}_\C$ action.

If all $\alpha_i$ are injective or all $\beta_i$ are surjective, then
the quiver is stable for the $H_\C$ action.
\end{remark}

\begin{remark}
As in \cite{DKS-Seshadri} we can look at (full flag) quivers where the
$\alpha_k, \beta_k$ are of the special `toric' form:
\begin{equation} \label{toric} 
\alpha_k = \left( \begin{array}{ccccc}
\nu_1^k & 0 & 0 & \cdots & 0\\
0 & \nu_2^k & 0 & \cdots & 0\\
 & & \cdots & & \\
0 & \cdots & 0 & 0 & \nu_k^k\\
0 & \cdots & 0 & 0 & 0 \end{array} \right) \end{equation}
and
\begin{equation} \label{toric2}
\beta_k = \left( \begin{array}{cccccc}
\mu_1^k & 0 & 0 & 0 & \cdots & 0\\
0 & \mu_2^k & 0 & 0 & \cdots & 0\\
 & & \cdots & & &  \\
0 & \cdots & 0 & 0 & \mu_k^k & 0  \end{array} \right) 
\end{equation}
for some $\nu^k_i, \mu_i^k \in \C$.
Now $\alpha_k \beta_k$ and $\beta_{k} \alpha_k$ are diagonal for each $k$,
so our quiver equations are just the diagonal components of (\ref{eq:mult}).
In fact they are equivalent to
\[
\mu_j^i \nu_j^i = q_i \ldots q_j -1.
\]
Note that $Y = 1 + \alpha_{n-1} \beta_{n-1}$ will also be diagonal.
\end{remark}

\bigskip
Let us now focus on the full flag case, so $r=n$ and $n_i =i$ for each $i$.
If all $\beta_i$ are surjective, then  we may use the $H_\C \times SL(n, \C)$
action to put $\beta_i$ in the standard form $\beta_i = (0 \; I_{i \times i})$. 
We now find that $Y= 1 + \alpha_{n-1} \beta_{n-1}$ lies in the standard Borel
of $GL(n, \C)$, with diagonal entries
\begin{equation} \label{scalars}
1, q_{n-1}, q_{n-1} q_{n-2}, \ldots, q_{n-1} \ldots q_1.
\end{equation}
Using the equations, and the fact that $\beta_i$ are in
standard form one may work down the quiver finding the $\alpha_i$
successively from $Y$. Conversely every such $Y$ arises from a solution of the
equations. The invertibility condition on the endomorphisms 
$1 + \alpha_i \beta_i$ and $1 + \beta_i \alpha_i$ is equivalent to the
scalars $q_j$ all being nonzero.

The freedom
involved in putting the $\beta_i$ in this form is the action of $N$,
conjugating $Y$ and acting on $SL(n, \C)$ on the right.

Our space of quivers (with all $\beta_i$ surjective) 
satisfying the equation modulo $H_\C$ is therefore
\[
SL(n, \C) \times_N B_1
\]
where $B_1$ denotes the subgroup of the Borel in $GL(n, \C)$
consisting of elements with $1$ as the leading term on the diagonal.
The geometric quotient $SL(n, \C) \times_N B_1$ can therefore
be viewed as sitting inside $Q_{\rm mult}$ (in the full flag case)
as a quasi-affine variety.

If we let $B$ denote the Borel in $SL(n, \C)$, then we have a degree $n$
cover $\rho: B \rightarrow B_1$ given by dividing by the leading diagonal
term. More explicitly, if the diagonal entries of an element in $B$ are
$z_1, \ldots, z_n$ and the diagonal entries of the corresponding
element $Y$ in $B_1$ are $w_1, \ldots, w_n$ then 
\[
w_1 = 1, \;\;\; w_i = \frac{z_i}{z_1} \; (i=2, \ldots,n), \;\;\; 
z_{1}^{n} = (w_2 \ldots w_n)^{-1} = (\det Y)^{-1}.
\]

\medskip
As in the additive case, we may generalise the above discussion to the
case of a general quiver with dimensions $n_1 < n_2 < \ldots < n_r =n$.
The space of such quivers with all $\beta$ surjective may be identified with
\[
SL(n, \C) \times_{[P,P]} {\mathcal P}
\]
where $P$ denotes the parabolic associated to the flag with dimensions
$(n_1, \ldots, n_{r}=n)$. Moreover $\mathcal P$ denotes the subvariety
of $SL(n,\C)$ consisting of matrices with scalar blocks down the diagonal,
of size $k_j \times k_j$ where $k_j = n_{j+1} - n_j$, and with all entries
below these blocks being zero. The scalars for
the blocks are those given by (\ref{scalars}).

In the full flag case when $n_i=i$ for each $i$, then the parabolic $P$
is the Borel, the variety $\mathcal P$ is $B_1$,
 and we recover the earlier result.

\begin{example} \label{SU2q}
Let us consider the $SL(2, \C)$ case, so our quiver is just 
\[
 \C \stackrel[\beta]{\alpha}{\rightleftarrows} \C^2.
\]
Our invertibility conditions are just equivalent to
\[
1 + a_1 b_1 + a_2 b_2 \neq 0
\]
where $\alpha = 
\left( \begin{array}{c}
a_1 \\
a_2 
\end{array} \right)$ and
$\beta = \left( b_1 \;\; b_2 \right)$. As $H_\C = SL(1, \C)$ is trivial there
are no moment map equations in this case, and no quotienting.
 So the quiver space
is just the complement in
 $\C^4$ of the hypersurface $1 + a_1 b_1 + a_2 b_2 = 0$.
\end{example}

It is useful to consider a slight modification of the quiver equations
so that we deal with the Borel $B$ in $SL(n, \C)$ rather than the
group $B_1$. We achieve this by setting 
\[
q_{i+1} = \frac{\tilde{q}_{i+1}}{\tilde{q}_i} \;\;\; : \;\;\; i=0, \ldots, r-2
\]
subject to the constraint $\tilde{q}_0 \ldots \tilde{q}_{r-1} =1$. Our equations
(\ref{eq:mult}) now become
 \begin{equation} \label{eq:multQ}
\tilde{q}_i (1 + \beta_{i+1} \alpha_{i+1}) = 
\tilde{q}_{i+1} (1 + \alpha_i \beta_i)  
\end{equation}
and recovering the $\tilde{q}_j$ from the $q_j$ involves choosing an $n$th 
root of unity. In terms of the matrix $Y = 1 + \alpha_{n-1} \beta_{n-1}$
introduced above, we have $\tilde{q}_{n-1}^n = \det Y$, so recovering
our solutions from $Y$ involves a choice
of an $n$th root of $\det Y$, as in the
above discussion of the cover $\rho : B \mapsto B_1$.
  We then obtain the geometric quotient $SL(n, \C) \times_N B$
as a moduli space of quivers with all $\beta_i$ surjective, sitting inside
the full quiver moduli space as an open dense subset.

As remarked in the Introduction, Boalch \cite{Boalch1}
has obtained a quasi-Hamiltonian $G \times T_C$ structure
on $G \times_N B$, for a general complex reductive group $G$.
(In fact he more generally obtains a quasi-Hamiltonian $G \times L$ structure
on the geometric quotient
 $G \times_U P$, where $P$
is a parabolic subgroup with unipotent radical $U$
and Levi factor $L$).

\section{Nonreductive GIT quotients}

We now make contact with nonreductive GIT quotients following \cite{DK}. The quotient $X/\!/G$,  in the sense of geometric invariant theory (GIT), of an affine variety $X$ over $\C$ by the action of a complex reductive group $G$ is the affine variety $\mathrm{Spec}(\mathcal{O}(X)^G)$ associated to the algebra $\mathcal{O}(X)^G$ of $G$-invariant regular functions on $X$. This makes sense because the algebra $\mathcal{O}(X)^G$ is finitely generated, since $X$ is affine and $G$ is reductive. If we want to quotient an affine variety by a nonreductive group then difficulties can arise because the algebra of invariants is not necessarily finitely generated.  However \emph{if} the algebra of invariants is finitely generated then we can define the GIT quotient to be the affine variety associated to this algebra, just as for reductive groups. 

It is worth noting that the inclusion of $\mathcal{O}(X)^G$ in $\mathcal{O}(X)$ induces a natural $G$-invariant morphism from $X$ to $X/\!/G$. When $G$ is reductive this morphism is always surjective, and points of $X$ become identified in $X/\!/G$ if and only if the closures of their $G$-orbits meet in $X$. However when the group is not reductive this morphism is not necessarily surjective; indeed its image is in general not a subvariety of the GIT quotient but only a constructible subset \cite{DK}.

Recall that the universal symplectic implosion for a compact group $K$ can be identified with the nonreductive GIT quotient $K_\C/\!/N$ of the complexified group $K_C$ (which is a complex affine variety) by the action of its maximal unipotent subgroup $N$ \cite{GJS:implosion}. Here the algebra of invariants $\mathcal{O}(K_\C)^N$ is finitely generated although $N$ is not reductive. In fact $K_\C/\!/N$ is the canonical affine completion of the quasi-affine variety $K_\C/N$, which embeds naturally as an open subset of $K_\C/\!/N$ with complement of codimension at least two. The restriction map from $\mathcal{O}(K_\C/\!/N)$ to $\mathcal{O}(K_\C/N)$ is thus an isomorphism, and both algebras can be identified with the algebra of $N$-invariant regular functions on $K_\C$. In terms of the moment map description of the symplectic implosion, $K_\C/N$ corresponds to the open subset determined by the interior of the positive Weyl chamber for $K$.

Recall also from \cite{DKS} that in the additive case when $K=\SU(n)$ the universal hyperk\"ahler
implosion can be identified with the GIT quotient 
$(K_\C \times \n^\circ) \symp N$
by the nonreductive group $N$. (On choosing an invariant
inner product, the annihilator $\n^\circ$ may be identified with
the opposite Borel subalgebra $\bmf$). Just as for the action of $N$ on $K_\C$,  the algebra of invariants turns out to be finitely generated, and the GIT quotient is defined to be the corresponding affine variety. As the moment map for the right $K_\C$
action on $T^*K_{\C}$ is projection onto the Lie algebra factor, the
quotient $(K_\C \times \n^\circ) \symp N$
can be viewed as the complex-symplectic quotient in the GIT sense of
$T^*K_\C$ by $N$.

A natural multiplicative version of this starts with the double
$K_\C \times K_\C$ instead of the cotangent bundle $T^*K_\C$ \cite{AMM}. We have
an action of $K_{\C} \times K_\C$, given by
\[
(u,v) \mapsto (g_L u g_R^{-1}, g_R v g_R^{-1})
\]
with quasi-Hamiltonian moment map
\[
(\mu_L, \mu_R) : (u,v) \mapsto (uvu^{-1}, v^{-1}).
\]

By analogy with the additive case we consider
the nonreductive GIT quotient
\[
(K_\C \times B ) \symp N
\]
where $N$ acts on the right
\[
(u,v) \mapsto (u n^{-1}, n v n^{-1})
\]

If $K = SU(n)$, 
the argument in the additive case can be adapted 
to the present situation to show that the nonreductive GIT quotient
$(SL(n,\C) \times B_1) \symp N$ may be identified with the space 
$Q_{\rm mult}$ of
solutions to the quiver equation (\ref{eq:mult}) in the full flag case, 
modulo (in the GIT sense) 
the action of $H_\C = \prod_{i=1}^{r-1} SL(i, \C)$. For one shows
that the 
resulting quiver variety is an affine variety with
coordinate ring equal to the coordinate ring 
$\cO (SL(n, \C) \times B_1)^N$ of the variety
of surjective quivers $SL(n, \C) \times_{N} B_1$. This identification of the coordinate ring is obtained by showing that 
$SL(n, \C) \times_{N} B_1$ is an open subset of the affine variety $Q_{\rm mult}$ with complement of codimension at least two.
So 
$Q_{\rm mult} = (SL(n,\C) \times B_1) \symp N$ may be viewed as the canonical 
affine completion of the geometric quotient $SL(n, \C) \times_{N} B_1$.

Working instead with (\ref{eq:multQ}) gives the 
analogous result for $(SL(n, \C) \times B) \symp N$
(recall that $B$ is an $n$-fold cover of $B_1$). Our quasi-Hamiltonian
reduction $(SL(n, \C) \times B) \symp N$ (which in general may be singular) 
may be thus viewed as the canonical affine completion of the
smooth quasi-Hamiltonian space $SL(n, \C) \times_{N} B$ as discussed at the
end of \S 3.

\begin{definition}
The multiplicative universal complex-symplectic implosion
for $SL(n, \C)$ is $\tilde{Q}_{\rm mult} = (SL(n, \C) \times B) \symp N$, or equivalently
the GIT quotient by $H_{\C}$ of the space of full flag quivers satisfying
(\ref{eq:multQ}).
\end{definition}

The left $SL(n, \C)$ action on $SL(n, \C) \times SL(n, \C)$ descends
to the nonreductive GIT quotient $\tilde{Q}_{\rm mult}= (SL(n, \C) \times
B) \symp N$. We also have a residual right action of $T_\C = B/N$.

The left moment map $\mu_L$ defined above is
 an $N$-invariant map $(u,v) \mapsto uvu^{-1}$ which descends
to a map $\tilde{Q}_{\rm mult} \rightarrow SL(n, \C)$. We also have a map
$\psi : \tilde{Q}_{\rm mult} \rightarrow T_\C$ given by projecting
onto the diagonal in $B$. 

If we take the level set $\psi^{-1}(1)$ and reduce by $T_\C$ we
obtain the affine variety 
$(SL(n, \C) \times N ) \symp B$. This is actually the target space of the 
multiplicative Springer resolution
\[
SL(n, \C) \times_B N \mapsto {\mathcal U};
\]
that is, it is the unipotent variety $\mathcal U$. The multiplicative
Springer map is just $\phi: (u,v) \mapsto uvu^{-1}$. The identification
of $\mathcal U$ with $(SL(n,\C) \times N) \symp B$ is the well-known
fact that the Springer map is an affinisation map.

More generally, we can reduce via $\psi$ at a level $\lambda$ in $T_\C$.
We obtain the quotient $SL(n, \C) \times \lambda. N \symp B$. 

Our map 
gives a surjection of $SL(n, \C) \times_B \lambda.N$ 
onto the {\em Steinberg fibre} $F_{\lambda}$ which is the
variety of elements in $SL(n,\C)$ where the regular class functions take the 
same values as they do on the diagonal matrix with entries $\lambda$.

The Steinberg fibres are the multiplicative analogues of the
Kostant varieties. We recall the following facts (see \cite{EL} or
\S 6 of \cite{St}, for example) that hold for general complex
semisimple $K_\C$ :

(i) each Steinberg fibre $F_{\lambda}$ is a finite union of conjugacy classes.
The dimension of $F_{\lambda}$ is $\dim K_\C - {\rm rank \;} K_\C$.

(ii) the regular elements form a single conjugacy class which is open and
dense. This class is the smooth locus of $F_{\lambda}$. Its complement
in $F_{\lambda}$ has complex codimension at least 2.

(iii) the semisimple elements in $F_{\lambda}$ form a single conjugacy 
class, the unique
closed class in $F_{\lambda}$. This class is contained in the closure of
each class in $F_{\lambda}$.

The map $\phi: K_\C \times_B \lambda. N \rightarrow F_{\lambda}$ 
is a resolution of singularities and is an isomorphism
over a locus in the target space whose complement
has codimension at least 2. As in the additive case, we conclude that
$F_{\lambda}$ is the affinisation $(K_\C \times \lambda.N) \symp B$.

So the reduction of $\tilde{Q}_{\rm mult} =(SL(n, \C) \times B) \symp N$ at level $\lambda$
gives the Steinberg fibre.

\begin{remark}
We can ask whether the Steinberg fibre could also be viewed as the reduction of $SL(n, \C) \times_B N$ by $T_\C$ at level $\lambda$ in the sense of GIT, since it is the affine variety associated with the appropriate algebra of invariant regular functions on $SL(n, \C) \times_B N$. However geometric invariant theory does not behave well when applied to actions on quasi-affine varieties such as $SL(n, \C) \times_B N$ which are not affine, since quasi-affine varieties are not determined by their algebras of regular functions even when these are finitely generated. For a reductive group action on a quasi-affine variety $X$ a categorical quotient of an open subset $X^{ss}$ of $X$ is given in \cite{GIT} Thm 1.10, but this differs in general from the affine variety associated to the algebra of invariants and $X^{ss}$ does not necessarily coincide with $X$, in contrast with the case when $X$ is affine. 
\end{remark}

\begin{remark} When $Y$ is a hyperk\"{a}hler manifold with an action of $SU(n)$ which is Hamiltonian in the hyperk\"{a}hler sense, its hyperk\"{a}hler implosion is constructed in \cite{DKS} as the hyperk\"{a}hler quotient of the product of $Y$ with the universal hyperk\"{a}hler implosion $Q$; it has an induced action of $T$ which may be complexified with respect to any of the complex structures to an action of $T_\C$. 
Likewise, given a general space with quasi-Hamiltonian $SL(n, \C)$ action, we may
take its product with $\tilde{Q}_{\rm mult}$ to get a space with
$SL(n, \C) \times SL(n, \C) \times T_\C$ action, and perform fusion  (cf. \cite{AMM})
to obtain a space with $SL(n, \C) \times T_\C$ action. Reducing by
$SL(n, \C)$ then yields a space with $T_\C$ action.
\end{remark}

\begin{remark}
  If, for a general semisimple $K_\C$, we could show finite generation
  of the ring of $N$-invariants ${\cO(K_\C \times B)}^N$, then the
  nonreductive GIT quotient $(K_{\C} \times B) \symp N$ would exist as an affine
  variety, and the discussion of this section would go through for
  general $K_\C$. The analogous result in the additive case is known
  by work of Ginzburg-Riche (\cite{GinzburgRiche}), but we have not
  yet been able to adapt it to the multiplicative setting.
\end{remark}

\begin{example}
Let us return to the $SL(2, \C)$ example.

We are considering the quotient $(SL(2, \C) \times B) \symp N$,
where as usual $B$ is the standard Borel and $N$ the 
associated maximal unipotent.

Let us write the elements of $SL(2, \C)$ and $B$ as 
$\left(\begin{array}{cc}
a & b \\
c & d
\end{array} \right)$ and
$\left(\begin{array}{cc}
e & f \\
0 & e^{\prime}
\end{array}
\right)$
with relations
\[
ad-bc =1 \;\;\; : \;\;\; e e^{\prime} =1.
\]
The action of $\left(\begin{array}{cc}
1 & n \\
0 & 1
\end{array} \right)$ leaves $a,c,e$ and $e^{\prime}$ invariant
and transforms $b,d,f$ as follows:
\begin{eqnarray*}
b & \mapsto & b - an \\
d & \mapsto & d -cn \\
f & \mapsto & f + n(e^{\prime} -e)
\end{eqnarray*}
The invariants are generated by $a,c,e,e^{\prime}$ and
\[
x := af + (e^{\prime} -e) b, \;\;\; : \;\;\;
y := cf + (e^{\prime} -e)d
\]
with relations
\[
cx - ay = e - e^{\prime} \;\; : \;\; e e^{\prime} =1.
\] 
In terms of new variables $X = ex$ and $Y = ey$ we can rewrite this
as
\[
cX - aY = e^{2}-1 \;\;\; : \;\;\; e \neq 0
\] 
that is, an open set in the complex quadric in $\C^5$. Note this
can also be written as the double cover of the complement in $\C^4$
of the locus $cX - aY = -1$, which is compatible with the quiver
picture as in Example \ref{SU2q}
\end{example}

\begin{remark}
If we take the element of $B$ to be in the real maximal torus,
that is we take $e = \exp(i \theta), e^{\prime} = \exp(-i \theta), f=0$, then
the relation becomes
\[
cx - ay = 2 i \sin \theta.
\]
If we take, for example, $c =i  \bar{x}, a = -i \bar{y}$, then
we get a copy of the group-valued symplectic implosion $S^4$
inside our variety. We have a copy of $S^3$ for each $\theta \in (0, \pi)$
and these collapse to a point at the endpoints $\theta =0, \pi$.
\end{remark}

We can generalise this idea to produce a copy of the quasi-Hamiltonian
symplectic implosion inside our complex space.

Let us take the $q_i$ to lie in the unit circle. We choose a branch
of the square root function on the half plane $\Re z < 0$, and 
consider quivers of the toric form (\ref{toric}), 
(\ref{toric2}) where the entries are given by

\begin{equation}
\nu_{j}^{i} = \mu_j^{i} = \sqrt{q_i \ldots q_j -1} 
\end{equation}

For such quivers we recall that $Y = 1 + \alpha_{n-1} \beta_{n-1}$
is diagonal with entries given by (\ref{scalars}).
Setting a consecutive run of entries of $Y$ to be equal is
equivalent to setting a consecutive run of $q_j$ to be $1$.
If, say, $q_i = \ldots q_{i+m}=1$ for some $m \geq 0$, 
(so that $m+2$ consecutive entries
of $Y$ are equal) then the last diagonal
entry of $\beta_i$, the last two of $\beta_{i+1}$, and so on
up to the last $m+1$ of $\beta_{i+m}$, are zero, and similarly
for the corresponding  $\alpha$. This means that
the quiver decomposes according to the direct sum
$\ker \alpha_j \oplus \im \beta_j$, where $\ker \alpha_j = \ker \beta_{j-1}$
and $\im \beta_j = \im \alpha_{j-1}$. Moreover the quiver maps are zero
on $\ker \alpha_j$ and injective and surjective on the complement.

Let us now consider the sweep of such quivers under the action of
$K = SU(n)$. Now the action of $K \times H_{\C} = SU(n) \times
\prod_{i=1}^{n-1} SL(i, \C)$ preserves the scalars $q_i$, hence
if $K$ moves a quiver to another quiver of the same form, then the
two quivers must be the same modulo the action of $H_{\C}$. That is,
$g \in SU(n)$ has the same effect on the quiver as $(h_2, \ldots, h_{n-1}) 
\in H_\C$. 
The resulting equations 
\[
\alpha_j  = h_{j+1} \alpha_j h_j^{-1} \;\;\; (2 \leq j \leq n-2),
\]
\[
\alpha_1 = h_2 \alpha_1, \;\;\;   \;\;\; 
\alpha_{n-1} = g^{-1} \alpha_{n-1} h_{n-1}^{-1},
\]
together with their analogues for $\beta$, now imply that $(h_2, \ldots, h_{n-1})$ and $g$
preserve the above decomposition. We get that $g$ lies in the commutator
of the parabolic associated to the dimension flag of the injective/surjective
quiver. More precisely, the dimensions of the injective/surjective quiver
are $1,2,\ldots, i-1, i-1, \ldots, i-1, i+m+1, .\ldots, n$ and we take the
parabolic associated to the strictly increasing 
sequence $1,2, \ldots, i-1, i+m+1, \ldots,n$
obtained by collapsing the chain of equalities.
But as $g$ lies in the maximal compact subgroup we find that
$g$ lies in $SU(m+2)$, diagonally embedded in $SU(n)$.

The sweep of our quivers is now $SU(n) \times_{SU(m+2)} \mathcal Y$
where $\mathcal Y$ denotes the set of diagonal matrices $Y$
with $Y_{i-1} = \ldots Y_{i+m}$.
This picture now generalises in the obvious way to the case 
of general systems of equalities between elements of $Y$.

We obtain a space stratified by sets $\mathcal Y$ of diagonal matrices
with entries satisfying specified equalities. Each such
face gives a copy of $SU(n) \times_{SU(n_1) \times \ldots
\times SU(n_r)} \mathcal Y$. The open interior face $\mathcal Y$ where all
entries are distinct just gives $SU(n) \times \mathcal Y$.

Finally, we pass to the cover. We work on the fundamental
alcove 
\[
\theta_1 - 2 \pi \leq  \theta_n \leq \theta_{n-1} \leq \ldots
\theta_2 \leq \theta_1
 \]
in the Cartan algebra $\sum_{i=1}^{n} \theta_i =0$, 
and consider the corresponding elements
$(e^{i\theta_1}, \ldots , e^{i \theta_n})$ of $SU(n)$. The
map $\rho$ sends this element to
$(1, e^{i (\theta_2 - \theta_1)}, \ldots,
e^{i (\theta_3 - \theta_2)})$ in $B_1$ . The fibres of the map are obtained by
multiplying by scalar matrices in $SU(n)$, but
within the alcove this means only the vertices map to the
same point (the identity in $B_1$). 

So upstairs in the covering space we obtain
a stratified space stratified by the walls of the alcove. The strata
are the sets 
\[
SU(n) \times_{SU(n_1) \times \ldots \times SU(n_r)} \mathcal Y
\]
discussed above.
The open stratum is $SU(n) \times \mathcal Y$, and at the other extremes
the vertices just give points. We have obtained the quasi-Hamiltonian
symplectic implosion introduced in \cite{HJS} as a subset of our
complex quasi-Hamiltonian implosion.

\begin{example}
In the case of $SU(2)$, we take the alcove $\theta_1 - 2 \pi \leq
\theta_2 \leq \theta_1$ in the Cartan algebra $\theta_1 + \theta_2=0$
and the associated elements $\rm{diag} (e^{i \theta_1},
e^{i \theta_2})$ of the maximal torus in $SU(2)$. Equivalently, we take
$0 \leq \theta \leq \pi$ and $\rm{diag} (e^{i \theta},e^{-i \theta})$.
The covering map to $B_1$ sends this to $(1, e^{-2i \theta})$ and is injective
in the given range except that the scalar matrices $\pm I$, 
corresponding to $\theta=0, \pi$ both map to the identity in $B_1$.

We obtain the quasi-Hamiltonian symplectic implosion by taking
an open stratum $SU(2) \times (0, \pi)$ and then adding point strata at the 
endpoints, to obtain $S^4$ as in \cite{HJS}.
\end{example}

\end{document}